\documentclass[12pt,reqno]{amsart}
\usepackage[utf8]{inputenc}
\usepackage{amsmath}
\usepackage{palatino}
\usepackage{xcolor}
\usepackage{amsfonts}
\usepackage{amsthm}
\usepackage{amssymb}
\usepackage{amscd}
\usepackage[all]{xy}
\usepackage{enumerate}
\usepackage{mathtools}
\theoremstyle{plain}

\newtheorem{thm}{Theorem}[section]

\newtheorem{prop}[thm]{Proposition}

\newtheorem{lem}[thm]{Lemma}

\theoremstyle{definition}

\newtheorem{remark}[thm]{Remark}

\newcommand{\mC}{\mathbb{C}}

\newcommand{\mP}{\mathbb{P}}

\newcommand{\mZ}{\mathbb{Z}}

\newcommand*{\de}{\partial}
\newcommand{\zfg}{Z_{f,g}}

\title{On the generic injectivity of Hessian maps \\
of ternary forms}
\author[Valentina Beorchia]{Valentina Beorchia$^{\circ}$}
\address[\textsc{Valentina Beorchia}]{University of Trieste,
Department of Mathematics, Informatics and Geosciences,
Via Valerio 12/1, 34127 Trieste, Italy; ORCID 0000-0003-3681-9045}
\email{beorchia@units.it}
\thanks{$^{\circ}$ The author is a member of ``Gruppo Nazionale per le Strutture Algebriche, Geometriche e le loro Applicazioni'', INdAM. She is partially supported by MUR funds: PRIN project GEOMETRY OF ALGEBRAIC STRUCTURES: MODULI, INVARIANTS, DEFORMATIONS, PI Ugo Bruzzo, Project code: 2022BTA242, and by the University of Trieste project FRA 2025. }

\thanks{{\bf 2020 Mathematics Subject Classification}{Primary: 14E05, 14H50,
14E20; Secondary: 14C15, 14C17, 14C20, 14C21, 14C25}}

\thanks{{\bf Keywords}. Plane curve, Hessian curve, birational map, Polar map, Chow ring,  ramification divisor, Jacobian linear system}

\begin{document}

\begin{abstract}
    We study the problem of the generic injectivity of the
    Hessian map, associating with a proportionality class of a ternary form the class of its Hessian determinant, conjectured
    by C. Ciliberto and G. Ottaviani in \cite{CO}. The conjecture has recently been proved in \cite{CO2}.
    
    Taking into account that the Hessian curve is the ramification divisor associated with the polar map,
    we perform a study of the problem using a geometric description of the graph of such a map.

\end{abstract}
\maketitle

\section{Introduction}
Hessian varieties associated with algebraic projective hypersurfaces are a very classical and rich topic, which is still an active area of research. For instance, the geometric properties of such varieties are, in general, far from being completely understood; results in this direction are given, for instance, in 
\cite{H},
\cite{DvG}, \cite{bfp} and \cite{bfp2}. 

 In this paper we investigate the question posed by C. Ciliberto and G. Ottaviani in \cite[Introduction, Question (ii), and Remark 6.4]{CO}, regarding the injectivity of the Hessian map for ternary forms.
 Specifically, given a projective class $[f] \in \mP (\mC [x_0,x_1,x_2])_d$ of a homogeneous polynomial $f$ of degree $d \ge 3$, we consider the rational map given by
 $$
    h_{d,2}: \mP(\mC [x_0,x_1,x_2]_d) \dasharrow \mP (\mC [x_0,x_1,x_2]_{3(d-2)}), \quad h_{d,2}([f])=[{\rm hess}_f], 
    $$
 where ${\rm hess}_f$ is the determinant of the Hessian matrix $M_f$ of $f$.

The indeterminacy locus of $h_{d,2}$ is the locus of polynomials with vanishing Hessian, which by Hesse Theorem coincides with the {\it cone locus}, that is polynomial classes corresponding to unions of concurrent lines. 

The case $d=3$ has been settled in \cite[Theorem 4.7]{CO}, where the authors prove that the Hessian map $h_{3,2}$ is dominant and generically $3:1$. The case $d \ge 4$ has been proved in \cite{CO2}.

 In this paper we deal with such a question.
 Our approach is geometric, and consists in analyzing the surface given by the graph in $\mP^2 \times \mP^2$ of the polar map of a smooth curve. 
 Such a map is finite of degree $(d-1)^2$, and the Hessian curve corresponds to the ramification divisor. By writing the numerical classes of the graph and of the ramification divisor on the graph, we can study the geometry of pairs of ramification divisors of some polar maps, which project to the same Hessian curve. If the two divisors coincide, we prove that the gradients of the two polynomials are proportional, hence by Euler formula the same holds for the starting polynomials (see Lemma \ref{lem: same ramification}). If the ramification divisors don't coincide, we consider a suitable ruled surface determined by them and, under the assumptions that such a surface is a product, we conclude that the two gradients are projectively equivalent, so that the two polar linear systems coincide. The case in which the ruled surface is not a product remains open.

\vskip 4mm
\medskip \noindent {\bf Acknowledgement.}
The author is grateful to K. Ranestad and A. Dimca for pointing out a mistake in a previous version. The author thanks also C. Ciliberto and G. Ottaviani for useful comments.

\section{Notation and Preliminaries}

Throughout the paper we shall indicate by
$$
T:= \mC [x_0,x_1,x_2].
$$
Given a homogeneous polynomial $f\in T_d$ of degree $d \ge 1$, we will denote by $V(f) \subset \mP^2$ the projective zero locus of $f$.
Moreover, with $[f]$ we will indicate the projective class of $f$ and by $\partial_i f$
the partial derivative 
$$
\partial_i f =\frac {\partial f}{\partial x_i}
$$

It is classically known that for a smooth curve the partial derivatives of $f$ are linearly independent and the polar map
$$
\nabla f : \mP^2 \to \mP (\langle \partial_0 f, \partial_1 f,\partial_2 f \rangle)
$$
$$
\nabla f (P) = [\partial_0 f (P) \partial_0 f + \partial_1 f(P) \partial_1 f + \partial_2 f(P) \partial_2 f]
$$
is a finite morphism of degree 
$(d-1)^2$, see for instance \cite[Section 1.2]{Dolgachev}.
The ramification divisor is given by the Hessian curve $H_f=V({\rm hess}_f)$ of degree $3(d-2)$, where ${\rm hess}_f$ is the determinant of the Hessian matrix, which we shall call $M_f$,
and the branch divisor $B_f$ has degree $3(d-1)(d-2)$.

Moreover, if $f$ is general enough, the Hessian curve is smooth and irreducible (see, for instance, \cite [p. 183]{EC}).
It is well known that if $V(f) \subset \mP^2$ is singular, then the singular locus of $H_f$ contains such points; hence the smoothness assumption on $H_f$ implies the smoothness of $V(f)$.

Finally, given $f\in T_d$, we shall denote its {\it polar linear system} by
$$
\Lambda_f := \langle \partial_0 f, \partial _1 f, \partial_2 f\rangle \subset \mP (T_{d-1}). 
$$
\section{Geometry of the polar map}
In this section we shall describe the geometry of the polar map of $f(x_0,x_1,x_2)\in T_d$, with $V(f)$ nonsingular and $H_f$ irreducible and smooth, through its graph. By fixing the isomorphism
$$
\varphi_f : \mP (\langle \partial_0 f, \partial_1 f,\partial_2 f \rangle) \to\mP^2, \qquad \varphi_f 
(a \partial_0 f+b \partial_1 f+ c\partial_2 f )=(a:b:c),
$$
we identify the polar map with
$$
\nabla f : \mP^2 \to \mP^2, \quad \nabla f(P) =
(\de_0f(P):\de_1 f(P):\de_2 f(P)).
$$
Then the graph $S_f$ is given by
$$
S_f =\{ (P, \nabla f (P)) \ | \ P \in \mP^2\}\subset \mP^2 \times \mP^2.
$$

The surface $S_f$ is irreducible and nonsingular, and it is isomorphic to $\mP^2$ via the restriction 
${p_1}_{|S_f}$ of the first projection $p_1 : \mP^2 \times \mP^2 \to \mP^2$. The equations of $S_f$ are given by the order $2$ minors of the following matrix:
\begin{equation}\label{eq: equations of Sf}
    S_f =\left \{(x_0:x_1:x_2)(y_0:y_1:y_2) \in \mP^2 \times \mP^2 \ | \ {\rm rk} 
    \left( \begin{array}{ccc}
    y_0 & y_1 & y_2 \\
    \partial_0 f & \partial_1 f & \partial_2 f \\
    \end{array}
    \right)
    =1 \right \},
\end{equation}
where $\partial _j f := \frac{\partial f}{\partial x_j}(x_0,x_1,x_2)$.

We shall denote by 
$$
  \rho_f:={p_2}_{|S_f}: S_f \to \mP^2 
  $$
  the degree $(d-1)^2$ finite morphism which lifts the polar map, and by $R_f \subset S_f$ the ramification divisor of $\rho_f$.

In what follows we shall determine the numerical classes of $S_f$ and of $R_f$.

We first set some notation. Let $A(\mP^2 \times \mP^2)$ be the Chow ring of $\mP^2 \times \mP^2$. By choosing $L_1$ and $L_2$ as generators of the Picard groups of the two factors, and by setting $p_i:\mP^2 \times \mP^2 \to \mP^2$ to be the two projections, we have that the two divisors
$h_1=p_1^\star L_1$ and $h_2=p_2^\star L_2$
are generators for $A(\mP^2 \times \mP^2)$. The following relations hold:
$$
h_1^3=0=h_2^3, \quad h_1^2 \cdot h_2^2 =1.
$$

\begin{prop}
    Let $V(f) \subset \mP^2$ be a smooth curve. Then the 
    graph $S_f \subset \mP^2 \times \mP^2$ of $\nabla f : \mP^2 \to \mP^2$ is a smooth surface with class
   \begin{equation}\label{eq: class of the graph}
S_f \equiv (d-1)^2 \ h_1^2 +(d-1) h_1 h_2 + h_2^2.
\end{equation}
Moreover, the ramification divisor $R_{f}$ satisfies
\begin{equation}\label{eq: ramification class}
R_{f} \equiv 3 (d-2) h_1 \ h_2^2 + 3(d-1)(d-2)h_1^2 \ h_2.
\end{equation}
\end{prop}

\begin{proof}
    Since $S_f$ is a codimension two cycle, its class can be written in the form
  $$
  S_f \equiv \alpha \ h_1^2 +\beta h_1 \ h_2 + \gamma h_2^2,
  $$  
  for some coefficients $\alpha, \beta, \gamma \in \mZ$. Being a graph, it is isomorphic to $\mP^2$ via the first projection $p_1$; in particular, $S_f\cdot h_1^2=1$, so
  $\gamma =1$. Moreover we have 
  $$
  \rho_f={p_2}_{|S_f}= \nabla f \circ {p_1}_{|S_f},
  $$
  so $S_f \cdot h_2^2=(d-1)^2 =\alpha$ is the degree of $\rho_f$. Finally, by the projection formula, we have 
  $$
  \beta =S_f  \cdot h_1 h_2=
  S_f  \cdot h_1 \cdot \rho_f^\star L_2={\rho_f}_\star (S_f \cdot h_1) \cdot L_2=
  $$
  $$
  = \nabla f_\star L_1 \cdot L_2 = d-1.
  $$
  Next we compute the ramification class, which is given by the relative canonical divisor:
  $$
  \begin{array}{rl}
  R_f & \sim K_{S_f}- \rho_f^\star K_{\mP^2}\\
   & \\
  & \sim (-3h_1)\cdot S_f -(-3h_2) \cdot S_f\\
   & \\
  & = -3(d-1)
h_1^2 \cdot h_2 -3 h_1 h_2^2+3(d-1)^2 h_1^2 \cdot h_2 +3(d-1)h_1 \cdot h_2^2\\
 & \\
& = 3(d-2)h_1 h_2^2 + 3(d-1)(d-2)h_1^2 h_2.
\end{array}
$$

\end{proof}

\begin{remark}
    In particular, the Hessian curve satisfies
$$
H_f={p_1}_\star R_f \sim 3(d-2)L_1,
$$
and the branch divisor $B_f$ satisfies
$$
B_f= {p_2}_\star R_f\sim 3(d-1)(d-2)L_2.
$$
\end{remark}

\begin{remark}
    Another lifting of the Hessian curve is given by the {\it Steinerian curve} $\Gamma_f
    \subset \mP^2 \times \mP^2$, defined by the equations $M_f \cdot \left ( \begin{array}{c}
    y_0 \\ y_1 \\ y_2\\ \end{array}\right)$, where $M_f$ is the Hessian matrix of $f$. Such a curve 
    is complete intersection of three divisors of class
    $(d-2)h_1 +h_2$, so its class is
    $$
    \Gamma_f \equiv 3(d-2)^2 h_1^2 h_2 + 3(d-2)h_1 h_2^2,
    $$
so it is not related to the ramification curve $R_f$. 

We note that some authors define the Steinerian curve as ${p_2}_\star \Gamma_f \subset \mP^2$.
\end{remark}
\section{Hessians and polar linear systems}

In order to analyze pairs of polynomials having the same Hessian curve, we shall
compare the graphs of their polar maps. To this aim, given
$V(f)$ and $V(g)$ two smooth curves of degree $d \ge 4$ ,
we fix the isomorphisms
$$
\varphi_f : \mP (\langle \partial_0 f, \partial_1 f,\partial_2 f \rangle) \to\mP^2, \varphi_f 
(a \partial_0 f+b \partial_1 f+ c\partial_2 f )=(a:b:c)
$$
$$
\varphi_g : \mP (\langle \partial_0 g, \partial_1 g,\partial_2 g \rangle) \to\mP^2, \varphi_g 
(a \partial_0 g+b \partial_1 g+ c\partial_2 g)=(a:b:c),
$$
 so that the two polar maps correspond to
$$
\nabla f , \ \nabla g : \mP^2 \to \mP^2, 
$$
$$
\nabla f(P) = 
(\de_0 f(P):\de_1 f(P):\de_2 f(P)), \ \nabla g(P) =(\de_0 g(P):\de_1 g(P): \de_2 g(P)).
$$
We shall denote by $S_f$ and $S_g$ the graphs of $\nabla f$ and $\nabla g$, respectively, and by $R_f$ and $R_g$ the corresponding ramification divisors. 

\begin{lem}\label{lem: same ramification}
    Assume that $V(f)\subset \mP^2$ and $V(g)\subset \mP^2$ are degree $d \ge 4$ smooth curves such that
    $$
    R_f =R_g.
    $$
    Then $V(g)=V(f)$.
    
\end{lem}

\begin{proof}
   By the assumption
   $R_f=R_g$, we have $(P,\nabla f(P))=(P,\nabla g(P)) \in \mP^2 \times \mP^2$ for any $P \in H_f$. Hence, for any $P \in H_f$, we have 
    $$
    {\rm rk}
    \left(
    \begin{array}{ccc}
    \partial _0 f (P) & \partial_1 f (P) &\partial_2 f (P) \\
    \partial_0 g (P) & \partial_1 g (P) &\partial_2 g (P) \\
    \end{array}
    \right)=1,
    $$
    so the order $2$ minors of the matrix 
    \begin{equation}\label{eq: matrix two gradients}
      \left(
    \begin{array}{ccc}
    \partial _0 f & \partial_1 f  &\partial_2 f  \\
    \partial_0 g  & \partial_1 g  &\partial_2 g \\
    \end{array}
    \right)  
    \end{equation}
   have $H_f$ as a common component. As such minors have degree $2d-2$ and $\deg H_f =3d-6 > 2d-2$ if $d \ge 4$, they are identically zero. It follows that $\nabla g = \lambda \nabla f$ for a suitable nonzero scalar $\lambda \in \mC$, and by the Euler identity we have $g =\lambda f$.
\end{proof}

Observe that the irreducibility of $H_f$ is not needed in the previous proof.

Next we will consider the case in which $R_f \neq R_g$.
We shall need the following result.

\begin{prop}\label{prop: two gradient map}
 Let $f,g\in T_d$ be two polynomials, with $V(f)\subset \mP^2$ smooth and $V(g)$ not a set of concurrent lines, and such that the three minors of the matrix
 \[
 \left(
    \begin{array}{ccc}
    \partial _0 f & \partial_1 f  &\partial_2 f  \\
    \partial_0 g  & \partial_1 g  &\partial_2 g \\
    \end{array}
    \right)  
 \]
 are linearly independent.

 Then the rational map $\nu_{f,g} : \mP^2 \dasharrow \mP^2$ defined by such a net, namely
 \[
 \nu_{f,g}(x_0,x_1,x_2)=  (\partial_0 f \, \partial_1 g - \partial_1 f \, \partial _0 g: \partial_0 f \, \partial_2 g - \partial_2 f \, \partial _0 g:\partial_1 f \, \partial_2 g - \partial_2 f \, \partial _1 g) ,
 \]
 is generically finite of degree $\le (d-1)^2$, and any contracted curve has degree $\le d-1$.
\end{prop}

\begin{proof}
By construction over any $Q=(q_0:q_1:q_2)\in \mP^2$, the fiber $\nu_{f,g}^{-1} (Q)$ is contained in the locus of equations
\[
\left\{
\begin{array}{ll}
q_0\partial _0 f +q_1 \partial_1 f  +q_2\partial_2 f & =0\\
q_0\partial _0 g +q_1 \partial_1 g  +q_2\partial_2 g & =0.\\
\end{array}
\right.
\]
As neither $V(f)$ nor $V(g)$ consist of concurrent lines, their partials are linearly independent, so far any $Q\in \mP^2$, the two equations are non-trivial. As a consequence,
a fiber consists either of a $0$-dimensional scheme of degree $ (d-1)^2$, or contains a curve of degree $\le d-1$ and a possible $0$-dimensional residual scheme of degree $< (d-1)^2$.
\end{proof}
\begin{prop}\label{thm: curves with same Hessian}
    Assume that $V(f)$ is a degree $d \ge 4$ curve with smooth Hessian $H_f$, and let $V(g)$ be a degree $d$ curve having the same Hessian curve: 
    $$
    H_f=H_g.
    $$
    In $\mP^2 \times \mP^2$ consider the complete intersection surface
    \begin{equation}\label{eq: zfg'}
    Z_{f,g}: \qquad 
   \left\{
    \begin{array}{l}
    {\rm hess} f =0 \\
      \\
     \det \left(
    \begin{array}{ccc}
    y_0 & y_1 & y_2 \\
    \partial _0 f & \partial_1 f &\partial_2 f \\
    \partial_0 g & \partial_1 g &\partial_2 g \\
    \end{array} 
    \right)=0,
        \end{array}
    \right.
    \end{equation}
    and let $\zfg'$ be the irreducible ruled surface, residual to the possible vertical components of type $p_1^\star P$, with $P\in H_f$.
    
    If $\zfg'\cong H_f \times \mP^1$ is a product, then
$$
    \langle \partial_0 f, \partial_1 f, \partial_2 f \rangle = \langle \partial_0 g, \partial_1 g, \partial_2 g \rangle \subset \mP(T_{d-1}).
    $$
\end{prop}
\begin{proof}
Since $H_g$ is smooth, this holds also for $V(g)$, so it is not a cone and its polar linear system is a net. By the assumption on the Hessians,
    we have
    $$
    H_f={p_1}_\star R_f={p_1}_\star R_g,
    $$
    and since $p_1$ is an isomorphism both when restricted to $S_f$ and $S_g$, we have that
    $R_g$ is smooth and irreducible too. Moreover, both
    the ramification curves are contained in the divisor ${p_1}^\star H_f \sim 3(d-2)h_1$ of equation ${\rm hess} f=0$.

    We next observe that the surface $Z_{f,g}$ is the closure of the union of lines spanned by the pairs of points $R_f \cdot p_1^\star P$ and $R_g \cdot p_1^\star P$, for $P \in H_f$. Indeed, the coordinates in $\mP^2 \times \mP^2$ of the two points $R_f \cap p_1^\star P$ and $R_g\cap p_1^\star P$ are  
    $$
    R_f \cap p_1^\star P=(P, \nabla f(P)), \quad R_g\cap p_1^\star P
    =(P, \nabla g(P)).
    $$
    Moreover we have $(P, \nabla f(P))=(P, \nabla g(P))$ for some $P \in H_f$ if and only if $(P, \nabla f(P)) \in R_f \cap R_g$.
    
    Hence the line joining the two points $\nabla f(P)$ and $\nabla g(P)$ in $p_1^\star P = \mP^2$ is given by 
    $$
    \det \left(
    \begin{array}{ccc}
    y_0 & y_1 & y_2 \\
    \partial _0 f (P) & \partial_1 f (P) &\partial_2 f (P) \\
    \partial_0 g (P) & \partial_1 g (P) &\partial_2 g (P) \\
    \end{array}
    \right)=0.
    $$
    
By the assumption $\zfg'$ is a product $H_f \times \mP^1$, we have that the three order two minors of the matrix
\[
\left(
    \begin{array}{ccc}
    \partial _0 f & \partial_1 f &\partial_2 f \\
    \partial_0 g & \partial_1 g &\partial_2 g \\
    \end{array} 
    \right)
\]
are proportional along $H_f$. Then the map $\nu_{f,g}$ considered in Proposition
\ref{prop: two gradient map} contracts $H_f$, so the three minors are linearly dependent.
Let
\[
a_0 (\partial_0 f \, \partial_1 g - \partial_1 f \, \partial _0 g) 
+ a_1 (\partial_0 f \, \partial_2 g - \partial_2 f \, \partial _0 g) 
+ a_2 (\partial_1 f \, \partial_2 g - \partial_2 f \, \partial _1 g) =0
\]
be a non trivial relation. This gives the following Jacobian syzygy of degree $2(d-1)$ for $f$:
\[
(a_0 \partial_1 g +a_1 \partial_2 g)\partial_0 f -(a_0\partial _0 g -a_2 \partial _2 g)
\partial_1 f -(a_1 \partial _0 g +a_2 \partial _1 g) \partial_2 f=0.
\]
Finally, since $V(f)$ is smooth, the three partials $\partial_0 f, \partial_1f$ and $\partial_2 f$ form a regular sequence, so the syzygy module of the Jacobian ideal is generated by the Koszul relations. In particular we have
\[
((a_0 \partial_1 g +a_1 \partial_2 g), -(a_0\partial _0 g -a_2 \partial _2 g), -(a_1 \partial _0 g +a_2 \partial _1 g) )=\]
\[
\qquad =b_0 (-\partial_1 f, \partial _0 f, 0) +b_1
(-\partial_2 f, 0, \partial_0 f) +b_2 (0, -\partial_2 f, \partial_1 f)
\]
for suitable scalars $b_0, b_1, b_2 \in \mC$. This implies that $\partial_0g, \partial_1 g,
\partial_2 g$ are linear combinations of the partials of $f$, and the statement follows.
\end{proof}

\subsection{Forms with equal polar linear systems}

Homogeneous forms with the same polar linear system have been classified, up to a suitable projective equivalence, by C. Mammana in
\cite{M}, with the following result:
\begin{thm}\label{thm: Mammana}
  Let $V(f) \subset \mP^n_\mC$ be a hypersurface of degree $d$, and denote by $M_f$ its Hessian matrix. If $V(f)$ is a not a cone, it has a polar linear system
  $\Lambda_f$ satisfying 
  $$
  \Lambda_f = \Lambda_g
  $$
  for some $V(g)\subset \mP^n_\mC$
  with $V(g) \neq V(f)$ if and only if $V(f)$ is projectively equivalent, via a matrix $A$ such that
  $A \cdot M_f$ is symmetric,
  to a hypersurface with equation of the type
  \begin{equation}\label{eq: Sebastiani Thom polynomials}
  f(x_0,\dots, x_n) = f_1 (x_0, \dots, x_h) + f_2 (x_{h+1}, \dots, x_n)
  \end{equation}
  for some $0 \le h \le n-1$,
  or
  \begin{equation}\label{eq: second type}
  f(x_0, \dots, x_n)=
 x_0 \partial_{h+1} \alpha  +
  x_1 \partial_{h+2} \alpha  + \dots + x_h \partial_{2h+1} \alpha + \beta,
  \end{equation}
  where $\alpha \in \mC[x_{h+1}, \dots, x_{2h+1}]_d $ and $\beta \in \mC[x_{h+1}, \dots, x_{n}]_d$, \  for some $0 \le h \le (n-1)/2$.
\end{thm}

\begin{remark}
    The polynomials of the type \eqref{eq: Sebastiani Thom polynomials}
 are called of {\it Sebastiani - Thom type} in \cite{Z}. Observe that such forms have a reducible Hessian hypersurface, and that the polynomials of type \eqref{eq: second type} for $n=2$ correspond to curves having a singular point of multiplicity $\ge d-1$.
 
 Moreover, as observed in \cite{ZErratum}, they don't form a closed family. However, their closure is a proper subset. Indeed, by
 the result of J. Carlson and Ph. Griffiths, see \cite [Section 4, (b)]{CG}, a general polynomial is uniquely determined by its Jacobian ideal.

 We note that the cone locus is contained in both closures of curves projectively equivalent to type \eqref{eq: Sebastiani Thom polynomials} and to type \eqref{eq: second type}.

 \end{remark}


\end{document}